\documentclass[12pt]{amsart}
 \usepackage[dvips]{graphicx}
 \usepackage{amscd,amsmath,amsthm,amssymb}
 \def\NZQ{\Bbb}               

 \def\ZZ{{\NZQ Z}}

 \def\KK{{\NZQ K}}
 
  \def\xx{{\bar{x}}}
 \def\frk{\frak}               

 \def\mm{{\frk m}}

 \def\F{{\mathcal F}}

 \def\opn#1#2{\def#1{\operatorname{#2}}} 
 \opn\chara{char} \opn\length{\ell} \opn\pd{pd} \opn\rk{rk}
 \opn\projdim{proj\,dim} \opn\injdim{inj\,dim} \opn\rank{rank}
 \opn\depth{depth} \opn\grade{grade} \opn\height{height}
 \opn\embdim{emb\,dim} \opn\codim{codim}
 
 \opn\Tr{Tr} \opn\bigrank{big\,rank}
 \opn\superheight{superheight}\opn\lcm{lcm}
 \opn\trdeg{tr\,deg}
 \opn\reg{reg} \opn\lreg{lreg} \opn\ini{in} \opn\lpd{lpd}
 \opn\size{size} \opn\sdepth{sdepth}
 \opn\link{link}\opn\fdepth{fdepth}\opn\lex{lex}
 \opn\div{div} \opn\Div{Div} \opn\cl{cl} \opn\Cl{Cl}
 \opn\Spec{Spec} \opn\Supp{Supp} \opn\supp{supp} \opn\Sing{Sing}
 \opn\Ass{Ass} \opn\Min{Min}\opn\Mon{Mon}
 \opn\Ann{Ann} \opn\Rad{Rad} \opn\Soc{Soc}
 \opn\Im{Im} \opn\Ker{Ker} \opn\Coker{Coker} \opn\Am{Am}
 \opn\Hom{Hom} \opn\Tor{Tor} \opn\Ext{Ext} \opn\End{End}
 \opn\Aut{Aut} \opn\id{id}
 
 \opn\nat{nat}
 \opn\pff{pf}
 \opn\Pf{Pf} \opn\GL{GL} \opn\SL{SL} \opn\mod{mod} \opn\ord{ord}
 \opn\Gin{Gin} \opn\Hilb{Hilb}\opn\sort{sort}
 \opn\aff{aff} \opn
 \con{conv} \opn\relint{relint} \opn\st{st}
 \opn\lk{lk} \opn\cn{cn} \opn\core{core} \opn\vol{vol}
 \opn\link{link} \opn\star{star}\opn\lex{lex}\opn\set{set}
 \opn\gr{gr}
 
 \def\pot#1#2{#1[\kern-0.28ex[#2]\kern-0.28ex]}

 \opn\dirlim{\underrightarrow{\lim}}
 \opn\inivlim{\underleftarrow{\lim}}
 \def\Implies{\ifmmode\Longrightarrow \else
         \unskip${}\Longrightarrow{}$\ignorespaces\fi}
 \def\implies{\ifmmode\Rightarrow \else
         \unskip${}\Rightarrow{}$\ignorespaces\fi}
 \def\iff{\ifmmode\Longleftrightarrow \else
         \unskip${}\Longleftrightarrow{}$\ignorespaces\fi}

 \let\:=\colon
 \newtheorem{Theorem}{Theorem}[section]
 \newtheorem{Lemma}[Theorem]{Lemma}

 \newtheorem{Example}[Theorem]{Example}

 \let\epsilon\varepsilon
 \let\kappa=\varkappa
 \def\qed{\ifhmode\textqed\fi
       \ifmmode\ifinner\quad\qedsymbol\else\dispqed\fi\fi}
 \def\textqed{\unskip\nobreak\penalty50
        \hskip2em\hbox{}\nobreak\hfil\qedsymbol
        \parfillskip=0pt \finalhyphendemerits=0}
 \def\dispqed{\rlap{\qquad\qedsymbol}}
 
 \opn\dis{dis}
 \def\pnt{{\raise0.5mm\hbox{\large\bf.}}}
 
 \opn\Lex{Lex}

\begin{document}
\title {A Koszul filtration for the second squarefree Veronese subring}
\author {Takayuki Hibi, Ayesha Asloob Qureshi and Akihiro Shikama}
\thanks{
\hspace{5.5mm}{\bf Keywords:}
second squarefree Veronese subring, Koszul algebra, Koszul filtration.
}
\address{Takayuki Hibi, Department of Pure and Applied Mathematics, Graduate School of Information Science and Technology,
Osaka University, Toyonaka, Osaka 560-0043, Japan}
\email{hibi@math.sci.osaka-u.ac.jp}
\address{Ayesha Asloob Qureshi, The Abdus Salam International Center of Theoretical Physics, Trieste, Italy}
\email{ayesqi@gmail.com}
\address{Akihiro Shikama, Department of Pure and Applied Mathematics, Graduate School of Information Science and Technology,
Osaka University, Toyonaka, Osaka 560-0043, Japan}
\email{a-shikama@cr.math.sci.osaka-u.ac.jp}

\begin{abstract}
The second squarefree Veronese subring in $n$ variables
is the edge ring of the complete graph with $n$ vertices.
It is proved that every second squarefree Veronese subring possesses
a Koszul filtration.
\end{abstract}
\maketitle

\section*{Introduction}
Inspired by the work on strongly Koszul algebras \cite{HeHiR}, Conca, Trung and Valla \cite{CTV} introduces the concept of Koszul filtrations.

Let $R$ be a standard graded algebra over a field $\KK$ with $\mm$ its graded maximal ideal.  A {\em Koszul filtration} of $R$ is a family ${\mathcal F}$ of ideals generated by linear forms with the properties as follows:
\begin{itemize}
\item
The zero ideal $ (0)$ and the maximal ideal $\mm$ belong to $\F$;
\item
for each $I\in {\mathcal F}$ with $I\neq 0$, there exists $J\in {\mathcal F}$ with
$J\subset I$ such that $I/J$ is a cyclic module
whose annihilator belongs to ${\mathcal F}$.
\end{itemize}

It is shown in \cite[Proposition 1.2]{CTV} that each ideal $I$ belonging to a Koszul filtration possesses a linear resolution.  In particular, any standard graded algebra over $\KK$ admitting a Koszul filtration is, in fact, Koszul. However, the example \cite[Page 101]{CRV} says that there exists a Koszul algebra with no Koszul filtration. It is known, for example see \cite[Corollary 6.6]{EH}, that if an ideal admits quadratic Gr\"obner basis with respect some suitable monomial order then its residue class ring is Koszul. However, the defining ideal of a Koszul algebra may not have quadratic Gr\"obner basis with respect to any monomial order. Also, in \cite{HHV} the authors give an example of a binomial edge ideal whose residue class ring has a Koszul filtration, while the ideal has no quadratic Gr\"obner basis with respect to any monomial order for the given labelling of associated graph. It is still unknown that whether the existence of a quadratic Gr\"obner basis of an ideal with respect to a suitable monomial order implies the existence of a Koszul filtration of its residue class ring. 

In the present paper, it is shown that every second squarefree Veronese subring possesses a Koszul filtration.  Recall that the {\em second squarefree Veronese subring} in $n$ variables is the subalgebra $S^{(2)}_{n} = \KK[\{ t_{i}t_{j} \}_{1 \leq i < j \leq n}]$ of the polynomial ring $S = \KK[t_{1}, \ldots, t_{n}]$. We identify $S^{(2)}_{n}$ with the {edge ring} (\cite{OH}) of the complete graph on $[n] = \{ 1, 2, \ldots, n \}$ and work with the {\em toric ideal}  of $S^{(2)}_{n}$, \cite{Stu}. It is known from \cite{Stu} that the toric ideal of $S^{(2)}_{n}$ admits quadratic Gr\"obner basis with respect to so called sortable monomial order.  However, apart from the
existence of such a quadratic Gr\'obner basis, a Koszul filtration of
$S^{(2)}_{n}$ studied in the present paper is also of interest from
viewpoint of combinatorics of finite graphs.  It seems very likely
that every squarefree Veronese subring possesses a Koszul filtration.


\section{Colon ideals in edge rings of complete graphs}
 Let $\KK$ be a field and $S=\KK[t_1, \ldots, t_n]$ be a polynomial ring in $n$ variables. Let $K_n$ be a complete simple graph on vertex set $[n]$ and $\KK[K_n]$ be the edge ring of $K_n$ generated by monomials $t_i t_j$ where $\{i,j\} \in E(K_n)$. The edge ring $E(K_n)$ can be viewed as second squarefree Veronese subring of $S$.  Let $T= \KK[x_{ij} :\{i,j\} \in E(K_n)]$. For the sake of convenience, we allow $x_{ij}=x_{ji}$. Let $I_n$ be the kernel of the K-algebra epimorphism $\phi : T \rightarrow \KK[K_n]$ with $\phi(x_{ij}) =  t_it_j$.

  To any even  closed walk  $w=\{i_1,i_2\}, \{i_2, i_3\}, \ldots, \{i_{j-1}, i_j\},\{i_j,i_1\}$ of length $j$, we associate a binomial $f_w = x_{i_1 i_2}x_{i_3 i_4}\ldots x_{i_{j-1}i_j} - x_{i_2 i_3}x_{i_4 i_5}\ldots x_{i_j i_1}$ in $T$. It is known
(\cite[Lemma 1.1]{OH} and \cite[Proposition 8.1.2]{V})
that the toric ideal $I_n$ of $\KK[K_n]$ is generated by quadratic binomials which correspond to 4-cycles in $K_n$. For any $f \in T$, we denote by $\bar{f}$ the residue class of $f$ in $T/I_n$.

 \begin{Theorem} \label{important}
 Let $K_n$ be a complete graph on vertex set $[n]$ and $\{ i,j\}, \{k,l\} \in E(K_n)$.
\begin{enumerate}
\item[{\em (a)}] Let $i=k$. Then  $(\bar{x}_{ij}) : \bar{x}_{il} = ( \{\bar{x}_{pj} :\;  p \neq l\})$.
\item[{\em (b)}] Let $\{ i,j\} \cap \{k,l\} = \emptyset $. Then  $(\bar{x}_{ij}) : \bar{x}_{kl} = (\bar{x}_{ij}, \{\bar{x}_{ip}\bar{x}_{jp}: p \neq  k,l)$.
\end{enumerate}
 \end{Theorem}

\begin{proof}
 (a): Let $p  \notin \{i,j,l\}$. Then the 4-cycle $\{ \{i,j\}, \{p,l\}, \{i,l\}, \{p,j\} \}$ in $K_n$ gives the binomial $x_{il}x_{pj} - x_{ij}x_{pl}$ in $I_n$ and hence $\bar{x}_{il}\bar{x}_{pj} \in (\bar{x}_{ij} )$. Therefore, $\bar{x}_{pj} \in (\bar{x}_{ij} ) :  \bar{x}_{il} $. Obviously,  $\bar{x}_{ij} \in (\bar{x}_{ij})  :  \bar{x}_{il}$. It gives $ ( \{\bar{x}_{pj} :\;  p \neq l\}) \subset (\xx_{ij}) : \bar{x}_{il} $.

Conversely, let $\bar{f} \in (\xx_{ij}) : \xx_{il} $ such that $\bar{f} \notin (\bar{x}_{ij})$. We may assume that $\bar{f}$ is a monomial in $T/I_n=\KK[K_n]$ because $\KK[K_n]$ is naturally equipped with $\ZZ^n$ grading with $\deg{\xx_{rs}} = \deg (t_r t_s) = \epsilon_r + \epsilon_s $ where $\{ \epsilon_1, \ldots, \epsilon_n \}$ are canonical basis for $\ZZ^n$. Then  $\bar{f} \xx_{il} = \bar{g} \xx_{ij}$ for some monomial $\bar{g} \in T / I_n$. Since  $\bar{f} \xx_{il} = \bar{f} t_it_l$ and $\bar{g} \xx_{ij} = \bar{g} t_i t_j$, we obtain that $t_j | f$ in $S$. It implies that $\xx_{pj} | \bar{f}$ (in $\KK[K_n]$). It remains to show that $\xx_{pj} | \bar{f}$ for some $p \neq l$.

Assume that $\xx_{pj} | \bar{f}$ only for $p = l$. Notice that $\xx_{jl} \notin (\xx_{ij}) : \xx_{il}$ because there does not exist any $\bar{v} \in T/I_n$ such that $\xx_{jl} \xx_{il} = \bar{v} \xx_{ij}$. It shows that $\bar{f}$ has more than one factors and there exist at least one factor $\xx_{rs}$ of $\bar{f}$ such that $r,s \neq l$. Indeed, if every factor of $\bar{f}$ is of the form $\xx_{rl}$ for some $r \in [n]$ then we may write $\bar{f}= \xx_{a_1l}\ldots \xx_{a_ml}$ for some $\{a_1, \ldots, a_m\} \subset [n]$. Then there does not exist any monomial $\bar{g} \in T/I_n $ which satisfy $\xx_{a_1l}\ldots x_{a_ml} \xx_{il} = \bar{g} \xx_{ij}$. Moreover, our assumption implies that $r,s \neq j$. If $r=i$, then $\xx_{is}\xx_{jl} = \xx_{ij} \xx_{sl}$ which gives $\bar{f} \in (\xx_{ij})$, a contradiction, and if $r ,s \neq i $ then $\xx_{rs} \xx_{jl}=\xx_{rj}\xx_{sl}$ which shows $x_{rj} | \bar{f}$ where $r \neq l$, again a contradiction.  Hence we conclude that $\xx_{pj} | \bar{f}$ for some  $p \neq l$.

(b): Let $p \notin \{i,j,k,l\}$. Then $\xx_{ip} \xx_{jp} \xx_{kl} = \xx_{ij} \xx_{kp} \xx_{lp}$. Therefore, $\bar{x}_{ip} \bar{x}_{jp} \in (\bar{x}_{ij} ): \bar{x}_{kl}$. Obviously,  $\bar{x}_{ij} \in (\bar{x}_{ij})  :  \bar{x}_{kl}$. It gives $ (\bar{x}_{ij}, \{\bar{x}_{ip}\bar{x}_{jp}: p \neq k,l\})  \subset  (\bar{x}_{ij}) : \bar{x}_{kl} $.

Conversely, let $\bar{f} \in (\xx_{ij}) : \xx_{kl}$ such that $\bar{f} \notin (\xx_{ij})$. As in the proof of (a), we may assume that $\bar{f}$ is a monomial in $T/I_n = \KK[K_n]$. Then  $\bar{f} \xx_{kl} = \bar{g} \xx_{ij}$ for some monomial $\bar{g} \in T/I_n$.  Since $\bar{f} \xx_{kl} = \bar{f} t_k t_l$ and $\bar{g} \xx_{ij} = \bar{g} t_i t_j$, we obtain that $t_it_j | \bar{f}$. It implies that $\xx_{pi} \xx_{qj} | \bar{f}$ for some $p,q \in [n]$. Suppose that $p \neq q$. Then $\xx_{ip} \xx_{jq} = \xx_{ij} \xx_{pq}$ and we get $\bar{f} \in (\bar{x}_{ij})$, a contradiction. This gives $p=q$. It remains to show that $\xx_{ip} \xx_{jp} | \bar{f}$ for some $p \neq  k,l$.

Assume that $\xx_{ik} \xx_{jk} | \bar{f}$. Notice that $\xx_{ik} \xx_{jk} \notin (\xx_{ij}) :\bar{x}_{kl}$ because there does not exist any monomial $\bar{g} \in T/I_n$ such that $\xx_{ik} \xx_{jk} \xx_{kl} = \bar{g} \xx_{ij}$.  It shows that $\bar{f}$ has more than one factors. Let $\xx_{ik}\xx_{jk}\xx_{rs} | \bar{f}$ for some $r,s \in [n]$. If $r=i$ then $s=k$, otherwise $\xx_{ik}\xx_{jk}\xx_{is} = \xx_{ik}\xx_{ij} \xx_{ks}$ which gives $\bar{f} \in (\xx_{ij})$, a contradiction. Also if $r=j$ then $s=k$ by similar reason.  We may choose $\xx_{rs}$ such that $r,s \notin \{i,j\}$, otherwise, every factor of $\bar{f}$ is of the form $\xx_{ik}$ or $\xx_{jk}$ but there does not exist any monomial $\bar{g} \in T/I_n $ which satisfy $\bar{f} \xx_{kl} = \bar{g} \xx_{ij}$. It also shows that we may choose $r,s \neq k$. Therefore, $\xx_{rs} \xx_{ik} \xx_{jk} = \xx_{ij} \xx_{rk} \xx_{sk}$ which gives $\bar{f} \in (\xx_{ij})$, a contradiction. Similar argument holds if we let $p=l$. Hence we conclude that $\xx_{ip} \xx_{jp} | \bar{f}$ for some $p \neq k,l$. This completes the proof.
 \end{proof}


 \section{Second squarefree Veronese subrings}

In order to define some suitable Koszul filtration for second squarefree Veronese subrings, we first introduce some notation. Let $G$ be a simple graph. We say that $G$ satisfies the {\em edge-distance} condition if for any $\{i,j\}, \{k,l\} \in E(G)$ with $\{i,j\} \cap \{k,l\} = \emptyset$, at least one of the edges $\{i,k\}, \{i,l\}, \{j,k\}, \{j,l\}$ belongs to $E(G)$.

Let $e \in E(G)$. We denote by $G\setminus \{e\}$ the subgraph of $G$ with edge set $E(G) \setminus \{e\}$ and the vertex set $ \cup_{f \in E(G\setminus \{e\})} V(f)$. For a vertex $v \in V(G)$, we denote by $N_G(v)$ the neighbour set of $v$ in $G$.

Before stating our next lemma, we first recall some definitions from graph theory. Let $G$ be a simple graph on vertex set $V(G)$ and edge set $E(G)$. Then $G$ is said to be {\em chordal} if every cycle of length greater than 4 in $G$ has a chord. Dirac's theorem on chordal graph states that a finite graph $G$ is chordal if and only if $G$ has a perfect elimination ordering, see \cite[p.172]{HHBook}. A graph $G$ on the vertex set $[n]$ is said to have a {\em perfect elimination ordering} if there exists an ordering $\{v_1, \ldots, v_n\}$ of vertices of $G$ such that each $v_i$ is simplicial in the subgraph induced by the vertices $\{v_1,\ldots, v_i\}$. A vertex is called {\em simplicial} if its neighbours induce a clique. For any $v \in V(G)$, elimination of $v$ from $G$ is defined by removing $v$ from $V(G)$ and the edges incident to $v$ from $E(G)$.


 \begin{Lemma}\label{subgraph}
Let $\F_n$ be the family of connected chordal subgraphs of a compete graph $K_n$ which satisfy the edge-distance condition. Then for any $H \in \F_n$ there exists an edge $e \in E(H)$ such that $H \setminus \{e\} \in \F_n$.
 \end{Lemma}

 \begin{proof}
 Let $H \in \F_n$ and $ \{v_1,v_2,\ldots, v_n\}$ be a perfect elimination ordering of $H$. We choose $e=\{v_n,i\} \in E(H)$ for some $i \in N_H(v_n)$ . Let $K = H \setminus \{e\}$. We first show that $K$ is a connected chordal graph. If the degree of $v_n$ is 1 then $\{v_1,\ldots, v_{n-1}\}$ is a perfect elimination ordering of $K$, and hence $K$ is a connected cordal graph. If the degree of $v_n$ is at least 2 then using the fact that $N_H(v_n)$ induces a clique it is obvious that $K$ is connected. Also, $v_n$ is a simplicial vertex of $K$. Indeed, $N_{K}(v_n) = N_H(v_n)\setminus \{i\} $ which implies $N_{K}(v_n)$ induces a clique. The ordering $ \{v_1,v_2,\ldots, v_n\}$ is a perfect elimination ordering of $K$ because $ \{v_1,v_2,\ldots, v_{n-1}\}$ induces the same subgraph in $K$ as in $H$. This shows that $K$ is chordal.

Now we show that $K$ satisfies the edge-distance condition. Let $\{p,q\},\{r,s\} \in E(K)$ with $\{p,q\} \cap \{r,s\} = \emptyset$. At least one of the edges  $\{p,r\}, \{p,s\}$, $\{q,r\}, \{q,s\}$ belongs to $E(H)$ because $H$ satisfies the edge-distance condition. Let $\{p,r\} \in E(H)$. If $\{p,r\} \in E(K)$, then we are done. Assume that $\{p,r\} \notin E(K)$. Since $E(H) \setminus E(K) = \{v,i\}$, it implies that $\{p,r\}= \{v,i\}$. Let $p=v$ and $r=i$. Then $q \in N_H(v)$. Using the fact that $v$ and its neighbour vertices in $H$ form a clique we obtain $\{q,i\} \in E(H)$. It implies that $\{q,i\} \in E(K)$, as required.
 \end{proof}

Let $H$ be a subgraph of $K_n$. Then we set $I_H=(\xx_{ij}: \{i,j\} \in E(H))$. Consider the following family of ideals:

\[
\F = \{I_H: H \in \F_n  \}.
\]

We now come to the main result of the present paper.

\begin{Theorem}
Let $K_n$ be a complete graph on vertex set $[n]$. Then $\F$ is a Koszul filtration of $\KK[K_n]$.
\end{Theorem}

\begin{proof}
First, note that $\emptyset \in \F_n$ and $I_{\emptyset} =(0) \in \F$.  To prove $\F$ is a Koszul filtration of $\KK[K_n]$, we have to show that for every $I_H \in \F$, there exist $I_K \in \F$ such that $I_H/I_K$ is cyclic and $I_K :I_H \in F$. From Lemma~\ref{subgraph}, we see that for any $I_H \in \F$, there exist $I_K \in \F$ such that $I_H/I_K$ is cyclic where $K = H \setminus \{v,i\}$ and $v$ is a simplicial vertex of $H$. It remains to show that $I_K : \xx_{vi} \in \F$.

We may write
\[
I_K=\sum_{p\in N_k(v)} (\xx_{vp}) + \sum_{q \in N_K(i)}( \xx_{iq}) + \sum_{\{r,s\} \in E(K) \atop r,s \notin \{v,i\}}(\xx_{rs} ).
\]

Then,

\[
I_K : \xx_{vi}= \sum_{p\in N_k(v)} (\xx_{vp}):\xx_{vi} + \sum_{q \in N_K(i)}( \xx_{iq}):\xx_{vi} + \sum_{\{r,s\} \in E(K) \atop r,s \notin \{v,i\}}(\xx_{rs} ):\xx_{vi}.
\]

By using Theorem~\ref{important} (a), we obtain

\[
\sum_{p\in N_k(v)} (\xx_{vp}):\xx_{vi} = \sum_{p\in N_k(v)} (\xx_{pr}: r \in V(K_n)\setminus \{i\}) ,
\]

\[
\sum_{q \in N_K(i)}( \xx_{iq}):\xx_{vi} = \sum_{q\in N_k(i)} (\xx_{ql}: l \in V(K_n)\setminus \{v\})
\]

and by using Theorem~\ref{important} (b), we obtain
\[
\sum_{\{r,s\} \in E(K) \atop r,s \notin \{v,i\}}(\xx_{rs} ):\xx_{vi}= \sum_{\{r,s\} \in E(K) \atop r,s \notin \{v,i\}} (\xx_{rs}, \xx_{rt}\xx_{st}: t \in V(K_n)\setminus \{v,i\}).
\]

As $K$ satisfies the edge-distance condition, we see that for every $\{r,s\} \in E(K_n)$ such that $r,s \notin \{v,i\}$, at least one of the edges $\{r,i\}, \{r,v\},\{s,v\}$ or $\{s,i\}$ of $E(K_n)$ belongs to $E(K)$. We know that $v$ is a simplicial vertex of $K$ which implies that $N_K(v) \subset N_K(i)$. It gives that either $r$ or $s$ belongs to $N_K(i)$. We may assume that $r \in N_K(i)$. Then

\[
(\xx_{rs} ):\xx_{vi} \subset ( \xx_{ir}):\xx_{vi}.
\]
which gives

\[
\sum_{\{r,s\} \in E(K) \atop r,s \notin \{v,i\}}(\xx_{rs} ):\xx_{vi} \subset \sum_{q \in N_K(i)}( \xx_{iq}):\xx_{vi}.
\]
Also, from $N_K(v) \subset N_K(i)$, we see that

\[
\sum_{p\in N_k(v)} (\xx_{vp}):\xx_{vi}  \subset  \sum_{q \in N_K(i)}( \xx_{iq}):\xx_{vi} + (\xx_{vp}: p \in N_K(v).
\]

It shows that

\begin{eqnarray}
\label{colon}
I_K : \xx_{vi} & = & \sum_{q \in N_K(i)}( \xx_{iq}):\xx_{vi} + \sum_{p\in N_k(v)} (\xx_{vp}):\xx_{vi} \\
 & = & \sum_{q\in N_K(i)} (\xx_{ql}: l \in V(K_n)\setminus \{v\}) + (\xx_{vp}: p \in N_K(v).\nonumber
\end{eqnarray}

From here we see that $I_K : \xx_{vi}$ is linearly generated. Let $J$ be the subgraph of $K_n$ with 

\begin{eqnarray}
\label{edge}
E(J) = \{ \{q,l\}: q \in N_K(i), l \in V(K_n)\setminus \{v\}\} \cup \{\{v,p\}: p \in N_K(v)\}
\end{eqnarray}

 and $V(J) = \cup_{e \in E(J)} V(e)$.  Then $I_J = I_K : \xx_{vi}$. We claim that $J \in \F_n$.  To prove our claim, we first notice that $K \subset J$ and for any $q \in N_K (i)$ we have $V(K_n)\setminus \{v,q\} \subset N_J(q)$. Also $\{v,i\} \notin E(J)$.

We first show that $J$ is connected chordal graph. Notice that $J$ is connected because for any $\{k,l\} \in E(J) \setminus E(K)$ either $k$ or $l$ belongs to $N_K(i)$. Now we show that $K_n$ is chordal. Suppose that there exists a cycle $C$ in $J$ with length greater than 3.  If $C$ is a cycle in $K$ then it has chord, and we are done. Otherwise, we may assume that $C$ is not a cycle in $K$ which implies that there exists an edge $\{t,u\} \in E(C)\setminus E(K)$. Then either $t$ or $u$ belongs to $N_K(i)$. Let $u \in N_K(i)$. Also, if $v \in V(C)$ then neighbours of $v$ in $C$ gives a chord because $v$ is simplicial vertex in $K$ and we are done. Let $v \notin V(C)$. Then from (\ref{edge}) we see that $u$ is incident to every vertex of $C$ in $J$, and hence $C$ has a chord. 


It remains to show that $J$ satisfies the edge-distance condition. Let $e,f \in E(J)$ with $e=\{p,q\}$ and  $f=\{r,s\}$ with $\{p,q\} \cap \{r,s\} = \emptyset$. If $e,f \in E(K)$, then we are done. Let $e \notin E(J)\setminus E(K)$. Then from (\ref{edge}), we see that either $p$ or $q$ belongs to $N_K(i)$. Let $p \in N_K(i)$. Then $\{p,r\}, \{p,s\} \in E(J)$ which shows that $J$ satisfies the edge-distance condition, as required.
\end{proof}

\begin{Example} {\em
For $K_5$, other than the empty subgraph, the connected chordal subgraphs of $K_n$  which satisfy the edge-distance condition (up to isomorphism) are given below. 
}
\end{Example}

\begin{figure}[htbp]
  \begin{center}
    \begin{tabular}{c}

      \begin{minipage}{0.14\hsize}
        \begin{center}
          \includegraphics[clip, width=1.8cm]{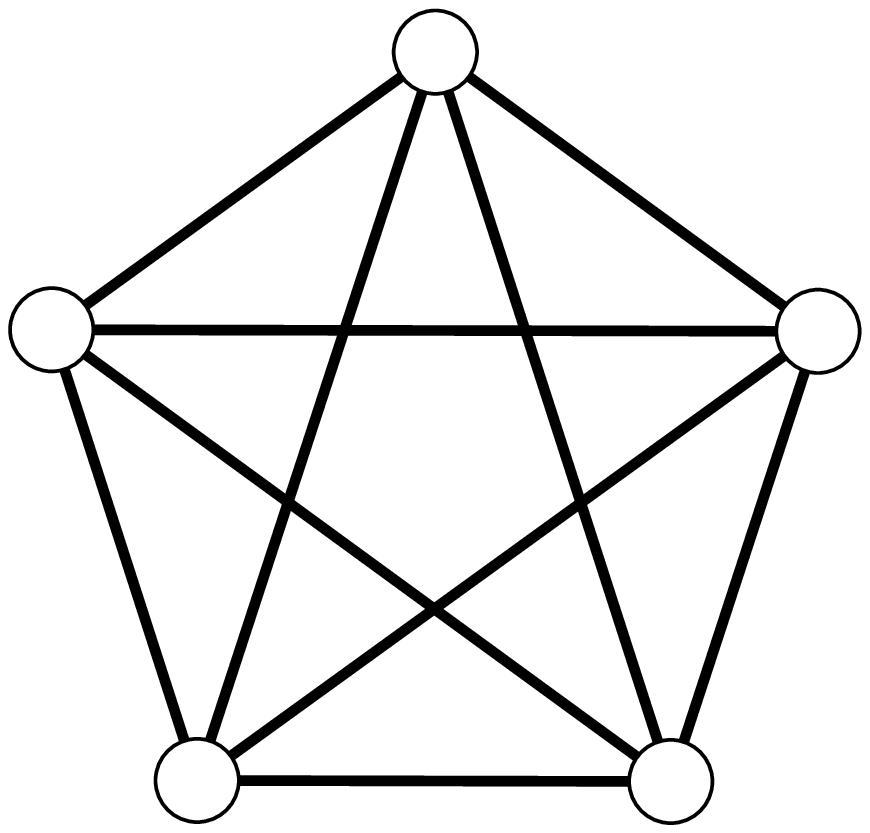}
        \end{center}
      \end{minipage}
      \begin{minipage}{0.14\hsize}
        \begin{center}
          \includegraphics[clip, width=1.8cm]{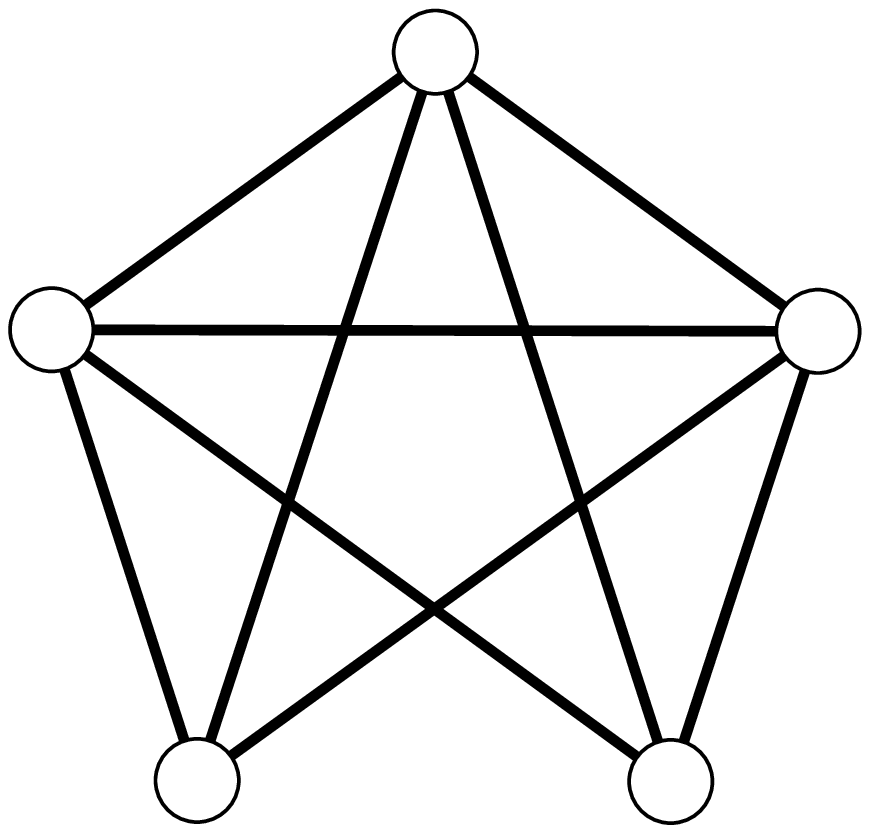}
        \end{center}
      \end{minipage}
      \begin{minipage}{0.14\hsize}
        \begin{center}
          \includegraphics[clip, width=1.8cm]{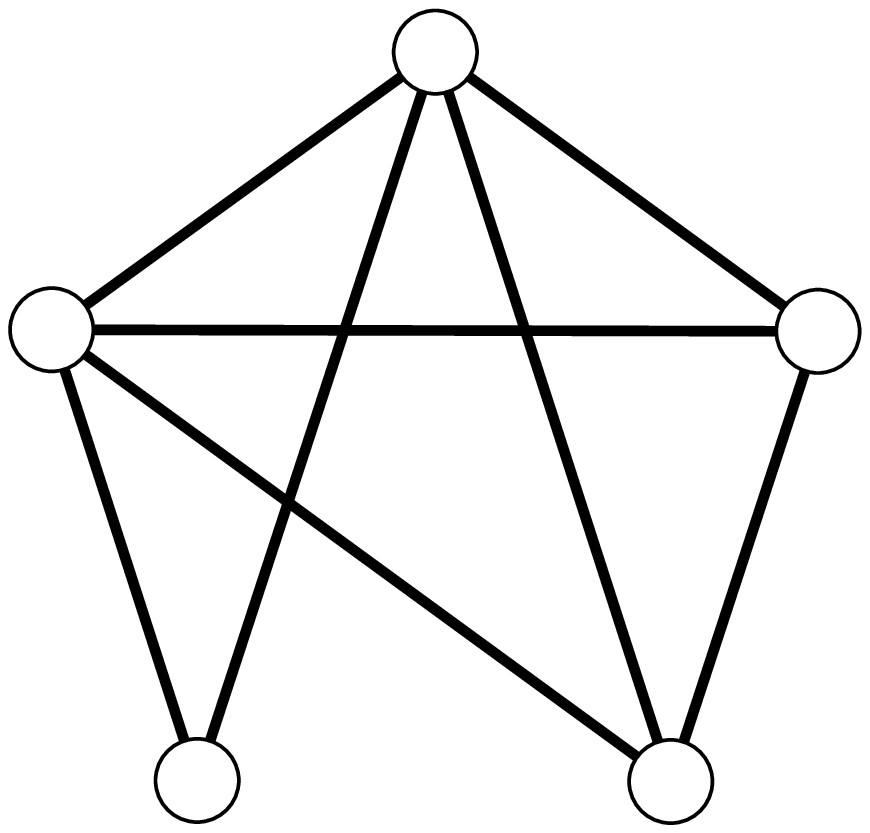}
        \end{center}
      \end{minipage}
      \begin{minipage}{0.14\hsize}
        \begin{center}
          \includegraphics[clip, width=1.8cm]{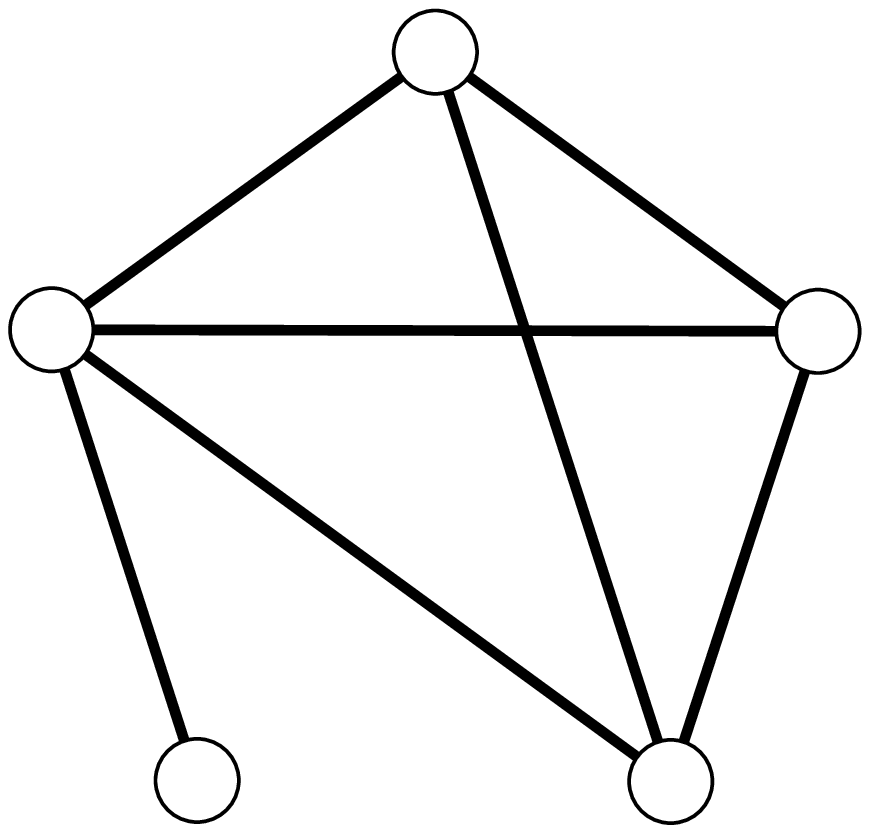}
        \end{center}
      \end{minipage}
      \begin{minipage}{0.14\hsize}
        \begin{center}
          \includegraphics[clip, width=1.8cm]{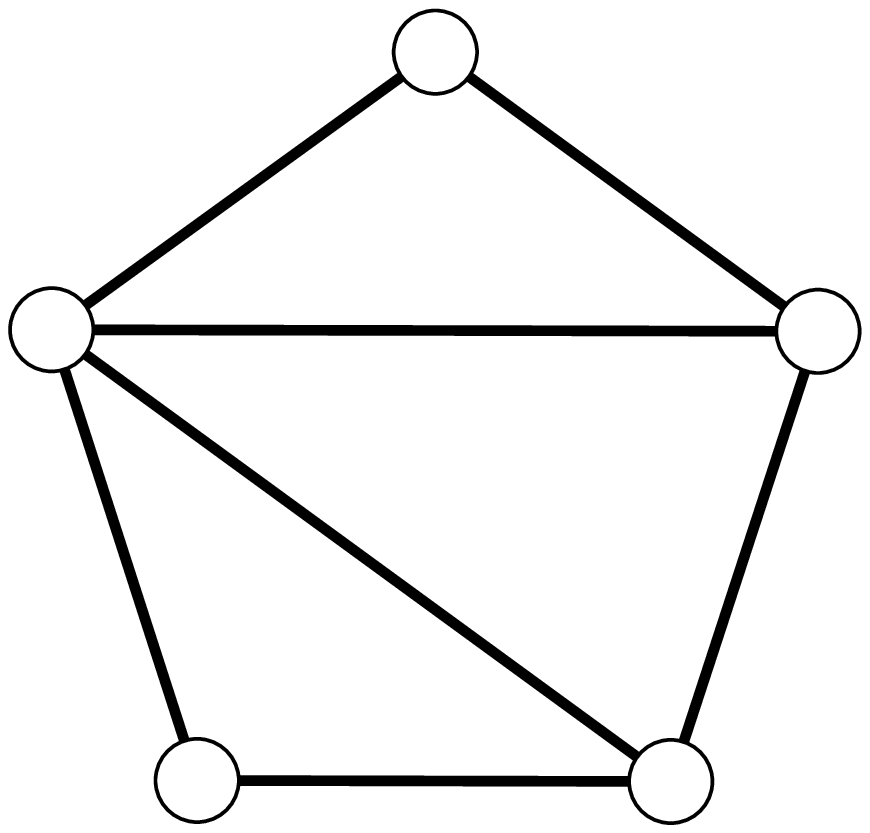}
        \end{center}
      \end{minipage}
      \begin{minipage}{0.14\hsize}
        \begin{center}
          \includegraphics[clip, width=1.8cm]{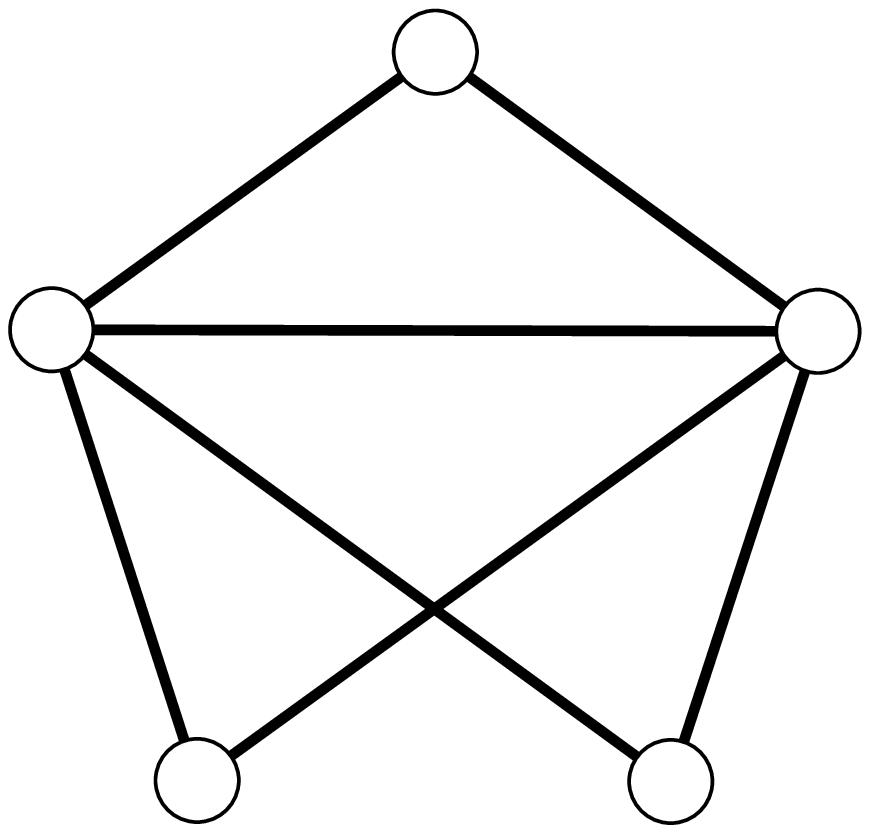}
        \end{center}
      \end{minipage}
      \begin{minipage}{0.14\hsize}
        \begin{center}
          \includegraphics[clip, width=1.8cm]{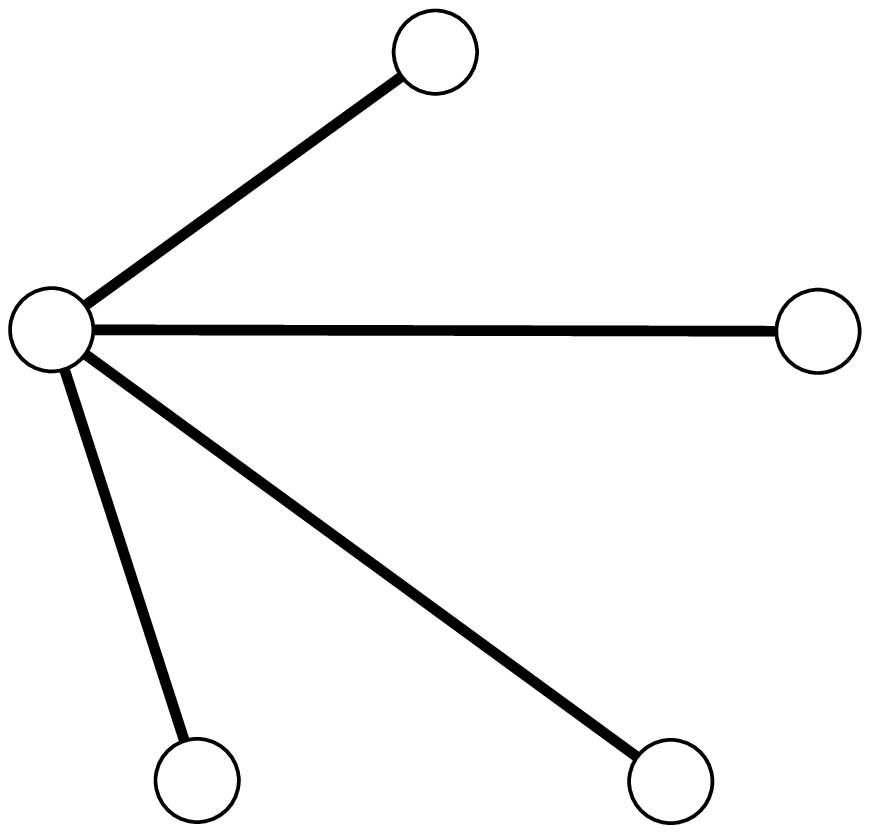}
        \end{center}
      \end{minipage}
    \end{tabular}

    \begin{tabular}{c}

      \begin{minipage}{0.14\hsize}
        \begin{center}
          \includegraphics[clip, width=1.8cm]{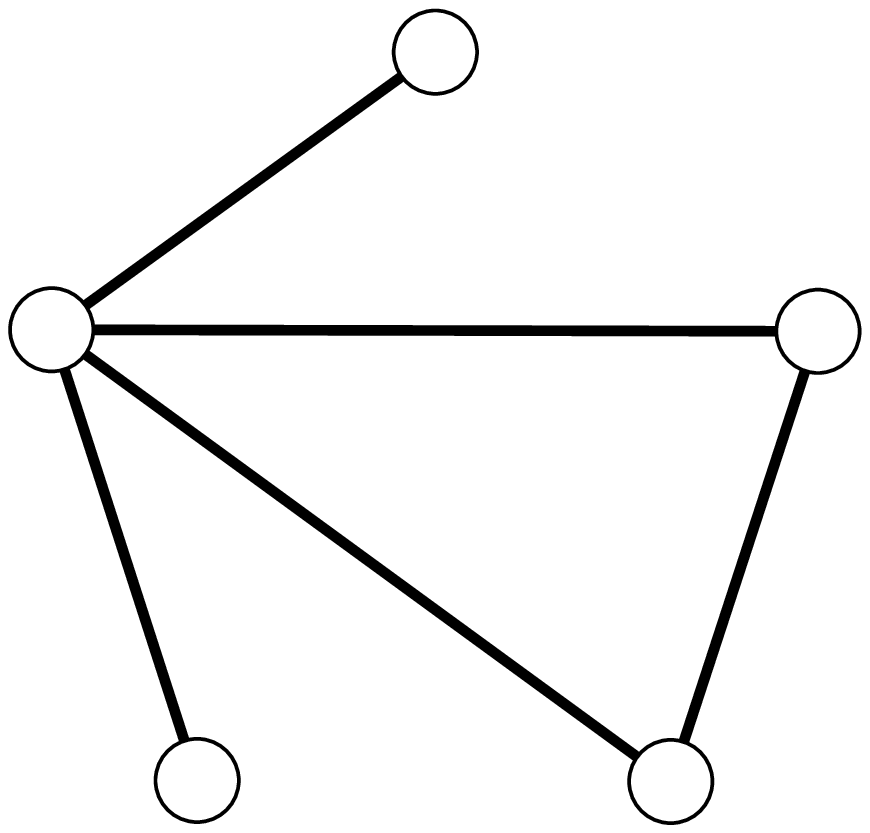}
        \end{center}
      \end{minipage}
      \begin{minipage}{0.14\hsize}
        \begin{center}
          \includegraphics[clip, width=1.8cm]{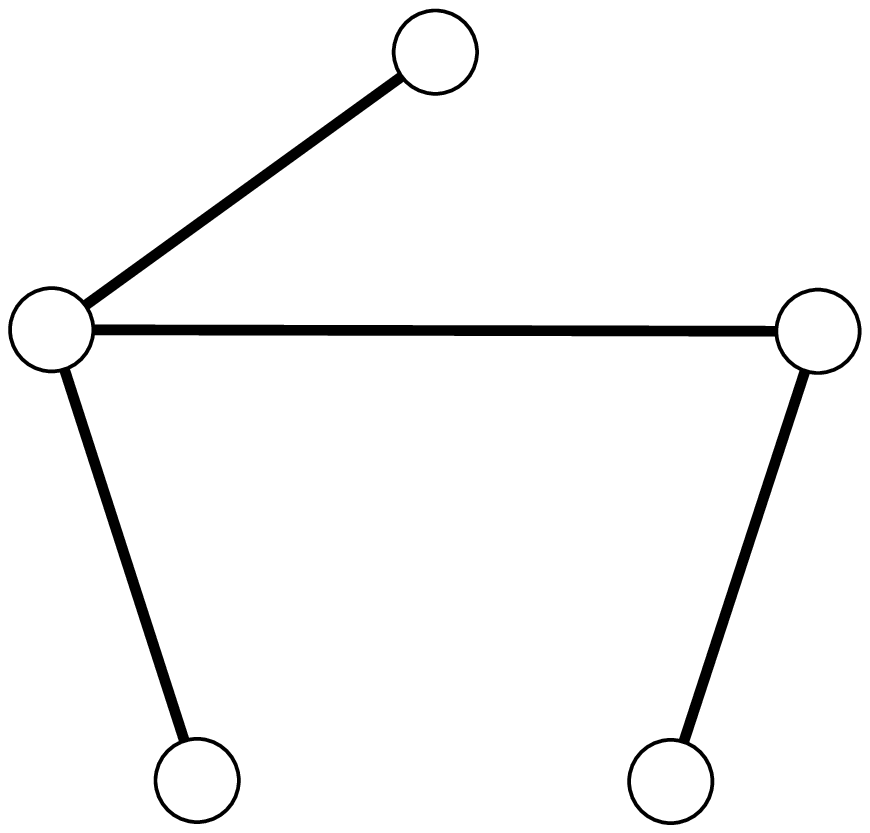}
        \end{center}
      \end{minipage}
      \begin{minipage}{0.14\hsize}
        \begin{center}
          \includegraphics[clip, width=1.8cm]{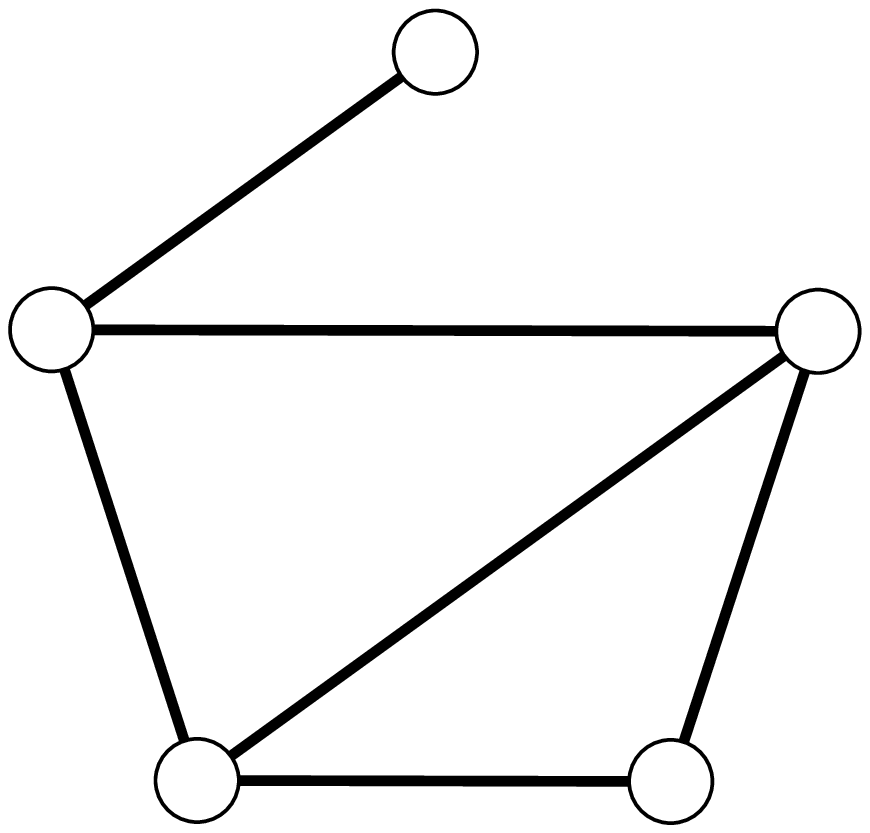}
        \end{center}
      \end{minipage}
      \begin{minipage}{0.14\hsize}
        \begin{center}
          \includegraphics[clip, width=1.8cm]{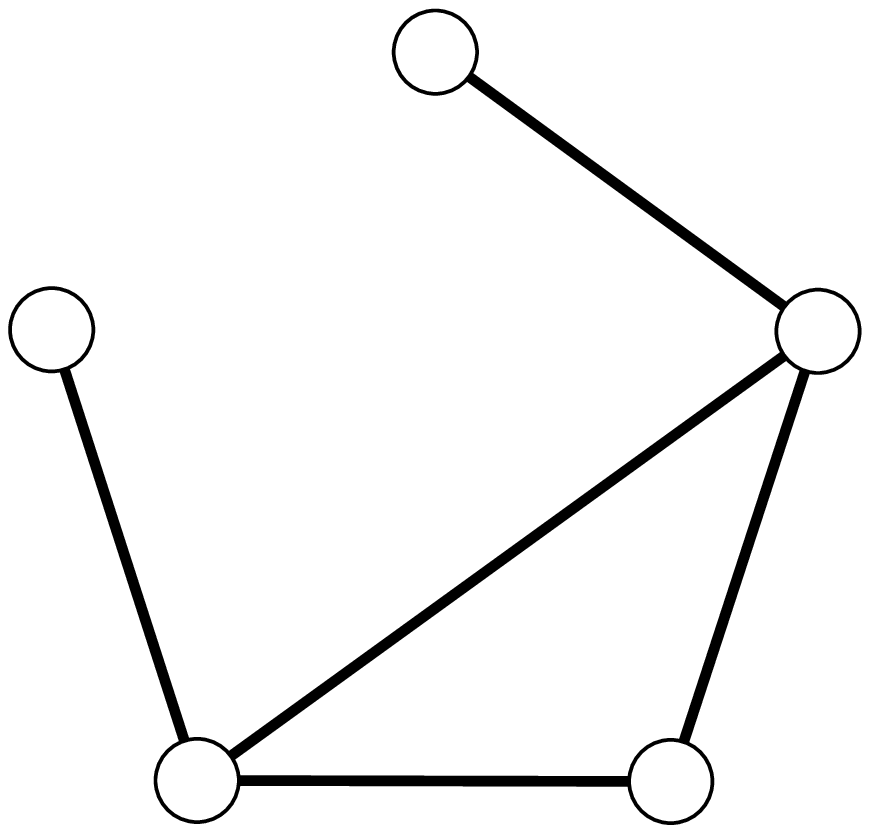}
        \end{center}
      \end{minipage}
      \begin{minipage}{0.14\hsize}
        \begin{center}
          \includegraphics[clip, width=1.8cm]{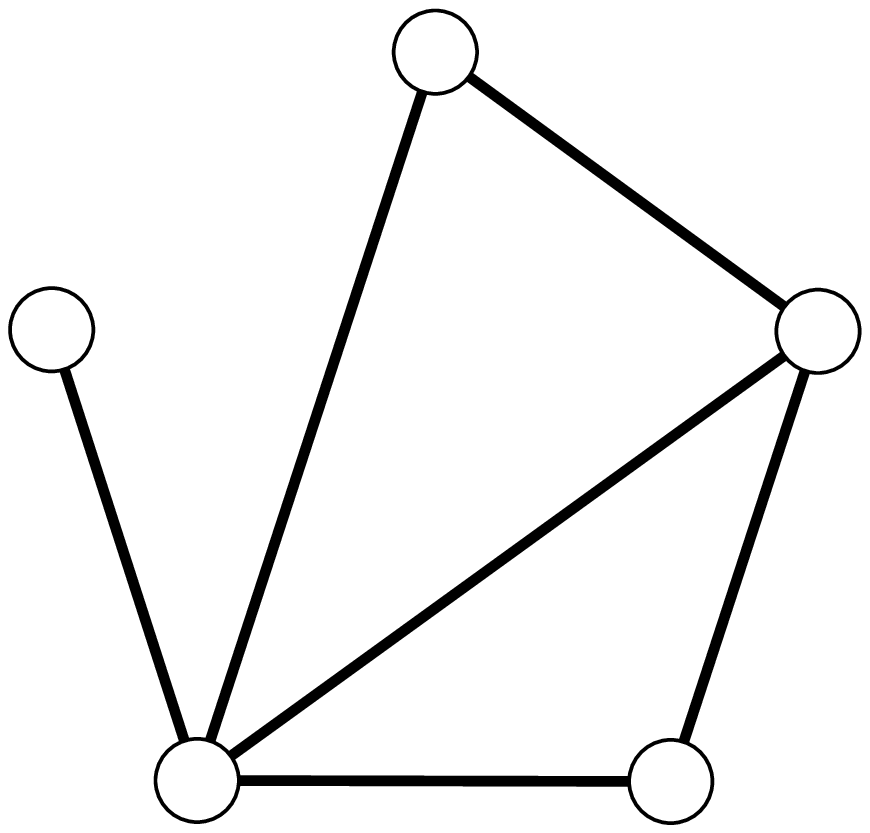}
        \end{center}
      \end{minipage}
      \begin{minipage}{0.14\hsize}
        \begin{center}
          \includegraphics[clip, width=1.8cm]{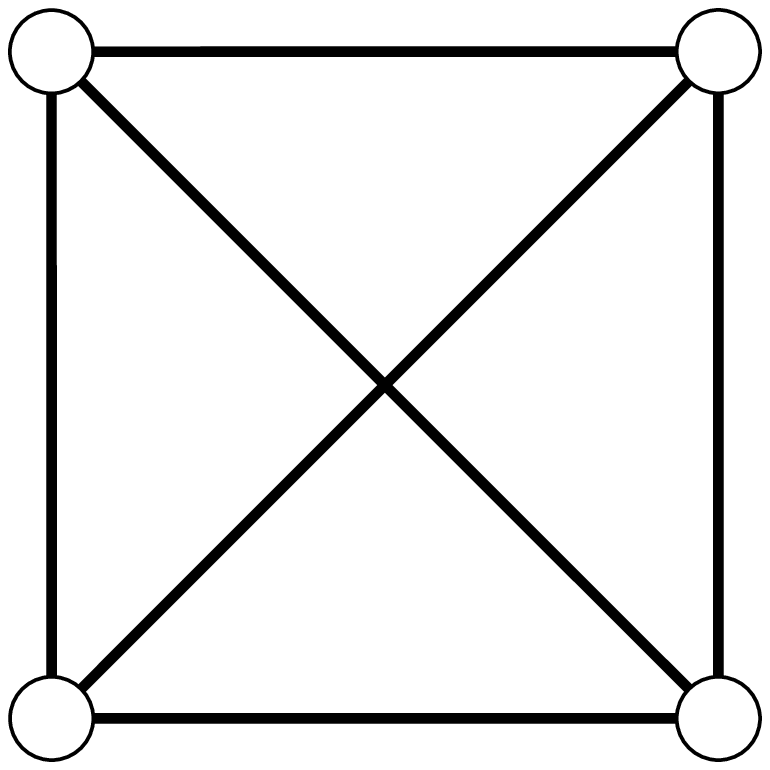}
        \end{center}
      \end{minipage}
      \begin{minipage}{0.14\hsize}
        \begin{center}
          \includegraphics[clip, width=1.8cm]{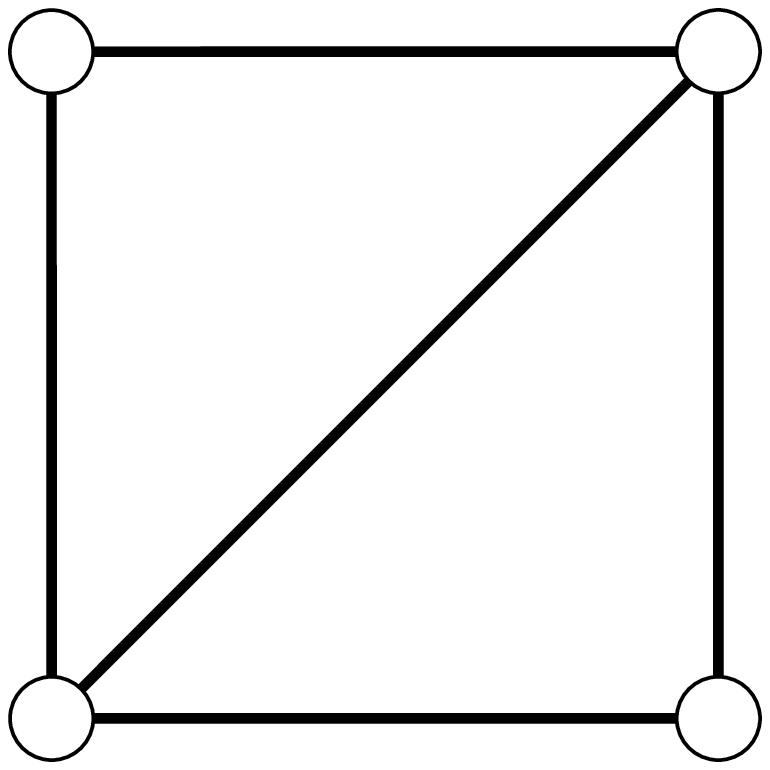}
        \end{center}
      \end{minipage}
    \end{tabular}

    \begin{tabular}{c}

      \begin{minipage}{0.14\hsize}
        \begin{center}
          \includegraphics[clip, width=1.8cm]{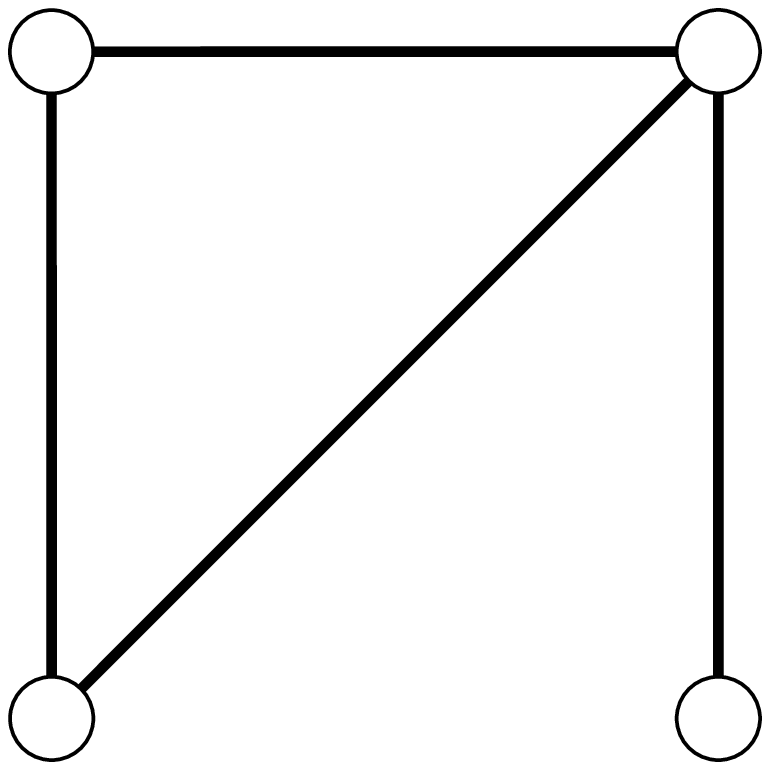}
        \end{center}
      \end{minipage}
      \begin{minipage}{0.14\hsize}
        \begin{center}
          \includegraphics[clip, width=1.8cm]{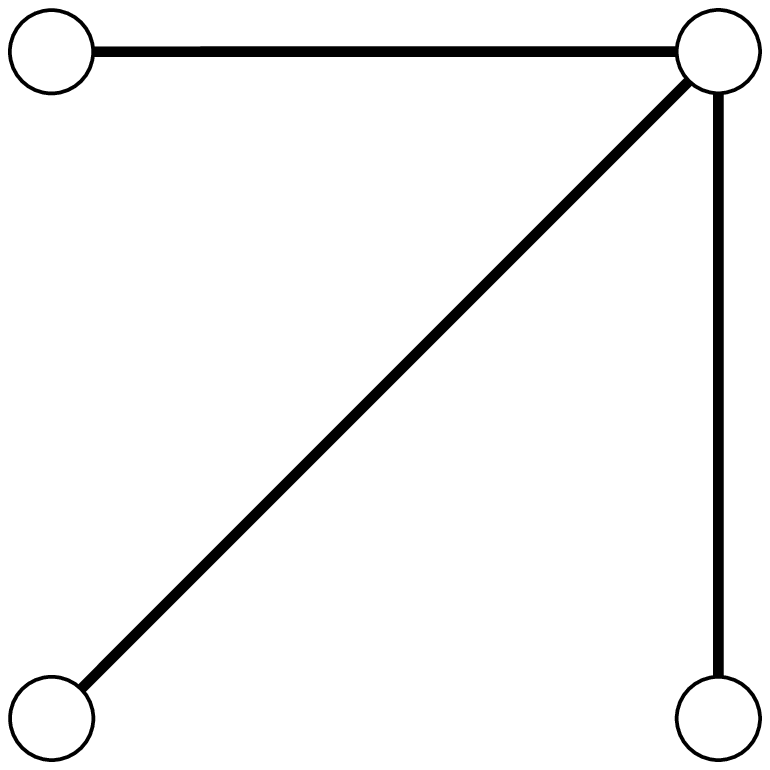}
        \end{center}
      \end{minipage}
      \begin{minipage}{0.14\hsize}
        \begin{center}
          \includegraphics[clip, width=1.8cm]{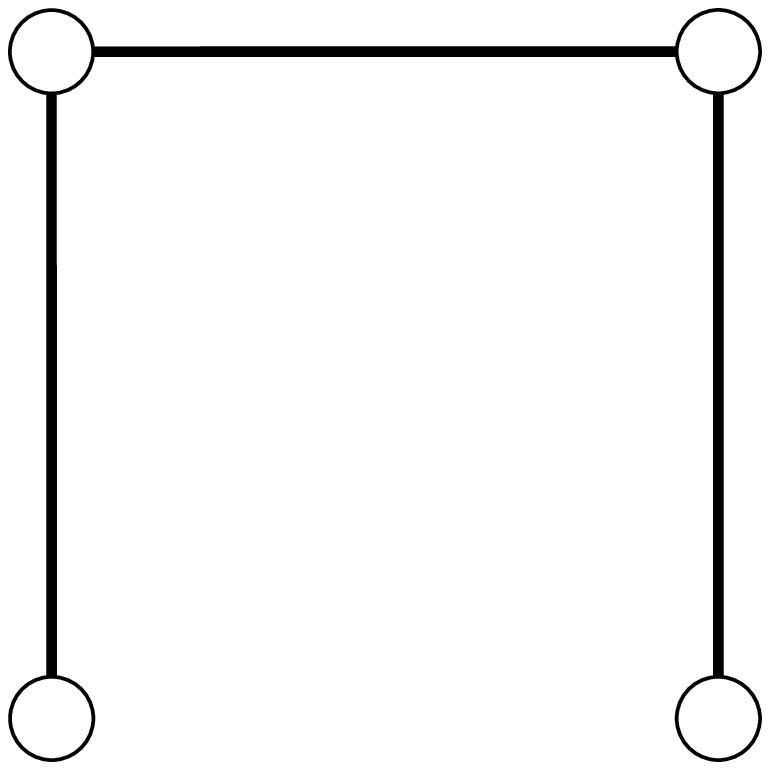}
        \end{center}
      \end{minipage}
      \begin{minipage}{0.14\hsize}
        \begin{center}
          \includegraphics[clip, width=1.8cm]{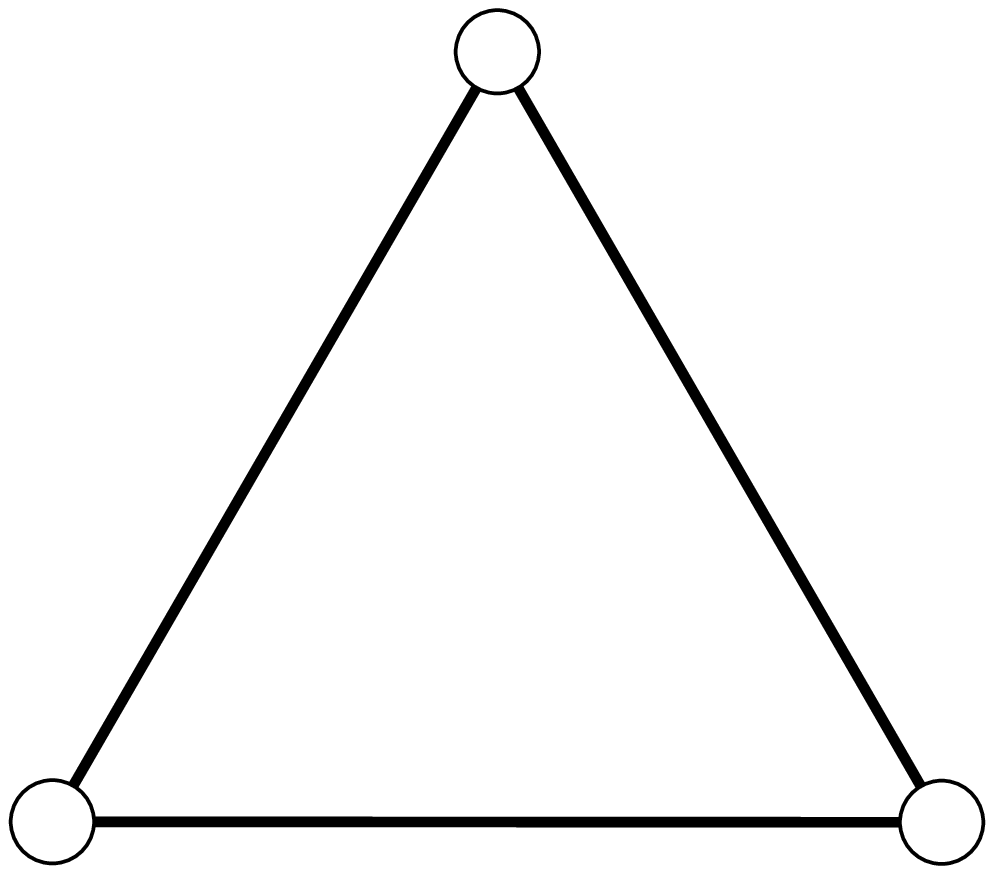}
        \end{center}
      \end{minipage}
      \begin{minipage}{0.14\hsize}
        \begin{center}
          \includegraphics[clip, width=1.8cm]{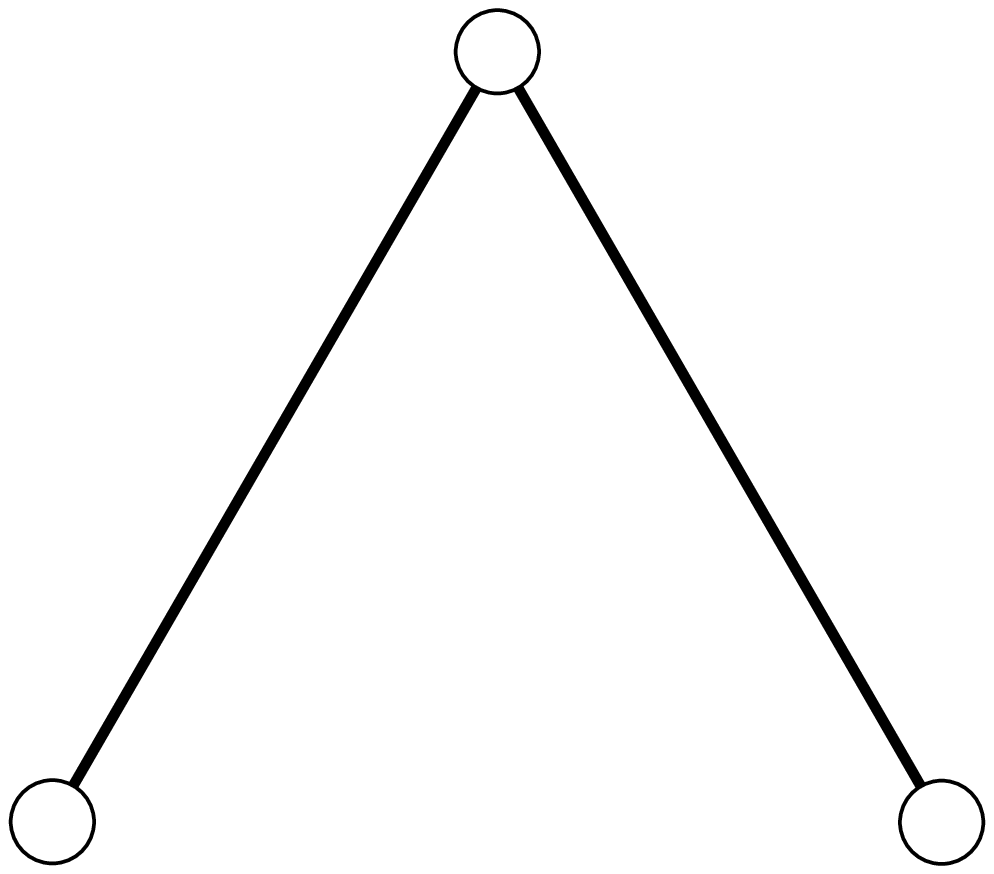}
        \end{center}
      \end{minipage}
      \begin{minipage}{0.14\hsize}
        \begin{center}
          \includegraphics[clip, width=1.8cm]{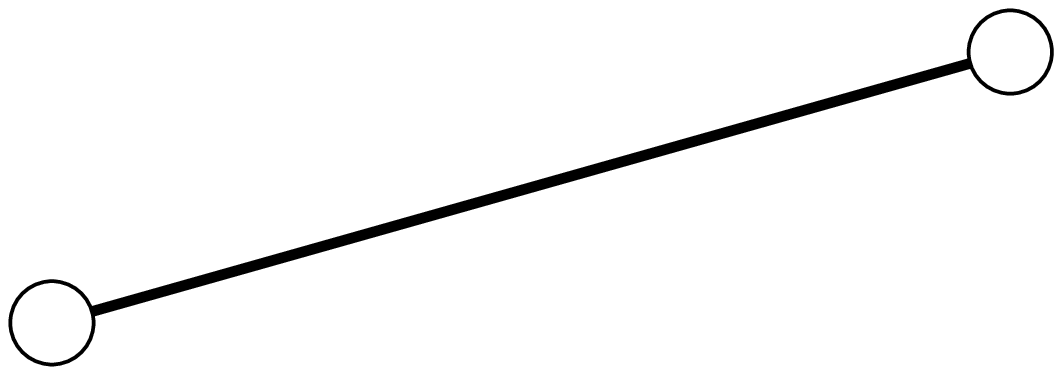}
        \end{center}
      \end{minipage}
      \begin{minipage}{0.14\hsize}
        \begin{center}
          \includegraphics[clip, width=1.8cm]{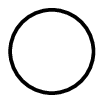}
        \end{center}
      \end{minipage}
    \end{tabular}

    \caption{}
  \end{center}
\end{figure}

\section*{acknowledgement}
This research was supported by the JST (Japan Science and Technology
Agency) CREST (Core Research for Evolutional Science and Technology)
research project Harmony of Gr\"{o}bner Bases and the Modern
Industrial Society in the framework of the JST Mathematics Program
``Alliance for Breakthrough between Mathematics and Sciences.''
{}
 \end{document}